\documentclass[12pt]{amsart}

\usepackage{amssymb}

\usepackage{enumerate}

\usepackage{graphicx}

\makeatletter
\@namedef{subjclassname@2010}{%
  \textup{2010} Mathematics Subject Classification}
\makeatother

\newtheorem{thm}{Theorem}[section]
\newtheorem{cor}[thm]{Corollary}
\newtheorem{lem}[thm]{Lemma}

\theoremstyle{definition}
\newtheorem{defn}[thm]{Definition}

\newtheorem{rem}[thm]{Remark}

\numberwithin{equation}{section}

\begin{document}


\baselineskip=17pt



\title[Additive complements for a given asymptotic density]{Additive complements for a given asymptotic density}

\author[A. Faisant]{Alain Faisant}
\address{Institut Camille Jordan\\ Universit\'{e} de Saint-\'{E}tienne\\
42023 Saint-\'{E}tienne Cedex 2, France}
\email{faisant@univ-st-etienne.fr}

\author[G. Grekos]{Georges Grekos}
\address{Institut Camille Jordan\\ Universit\'{e} de Saint-\'{E}tienne\\
42023 Saint-\'{E}tienne Cedex 2, France}
\email{grekos@univ-st-etienne.fr}

\author[R. K. Pandey]{Ram Krishna Pandey}
\address{Department of Mathematics\\ Indian Institute of Technology\\
Roorkee--247667, India}
\email{ramkpfma@iitr.ac.in}

\author[S. T. Somu]{Sai Teja Somu}
\address{School of Mathematics\\ Tata Institute of Fundamental Research\\
Mumbai--400005, India}
\email{somuteja@gmail.com}

\date{April, 9, 2019}

\begin{abstract}
{\bf The first version of this text was written and submitted to a journal on April, 12, 2018. This second version was submitted on April, 9, 2019.} We investigate the existence of subsets $A$ and $B$ of $\mathbb{N}:=\{0,1,2,\dots\}$ such that the sumset $A+B:=\{a+b~;a\in A,b\in B\}$ has given asymptotic density. We solve the particular case in which $B$ is a given finite subset of $\mathbb{N}$ and also the case when $B=A$ ; in the later case, we generalize our result to $kA:=\{x_1+\cdots+x_k: x_i\in A, i=1,\dots,k\}$ for an integer $k\geq2.$
\end{abstract}



\maketitle

\section{Introduction}
The purpose of this paper is to introduce a new, up to our knowledge, subject of research, resolving a few particular cases.

Two subsets, not necessarily distinct, $A$ and $B$ of $\mathbb{N}:=\{0,1,2,\dots\}$ are called {\it additive complements} if $A+B$ contains all, except finitely many, positive integers; that is, if the set $\mathbb{N}\setminus(A+B)$ is finite. From the huge literature on this topic, we just mention the papers \cite{L}, \cite{Nar} and \cite{D} among the initial ones, and the papers \cite{FC} and \cite{R} among the last ones. The reader may find there information and more references. A set $A$ being additive complement of itself, that is a set $A$ such that $\mathbb{N}\setminus(2A)$ be finite (where we put $2A:=A+A),$ is called an {\it asymptotic basis} of {\it order} 2.

In this paper we are interested in what happens when one asks that the density of the sumset $A+B$ is equal to a given value $\alpha,$ $0\leq\alpha\leq 1.$ As density concept we use the asymptotic density, defined below.
\begin{defn}\label{definition1}
Let $X$ be a subset of $\mathbb{N}$ and  $x$ a real number. For $x\geq1,$ we put $X(x):=|X\cap[1,x]|.$ For $x<1,$ we put $X(x)=0.$  We define the {\it asymptotic} (also called {\it natural}) {\it  density} of $X$ as
$$dX:=\lim_{x\rightarrow+\infty}{{X(x)}\over x}$$
provided that the above limit exists. The {\it lower} and the {\it upper asymptotic densities}, denoted by
${\underline d}X$ and ${\overline d}X,$ respectively, are defined by taking in the above formula the lower and the upper limits, respectively, which always exist.
\end{defn}

In the classical situation the existence of additive complements is obvious and the main problem is to find ``thin" subsets of $\mathbb{N}$ being additive complements or asymptotic bases. Our first question in this study is to guarantee the existence of such sets $A$ and $B$. Precisely, in this paper, we establish the existence of $A$ for a given finite set $B$. Furthemore, we also consider the case $B = A$, in this existence.  Once the existence has been established, two questions naturally arising are to find {\it thin} (similarly {\it thick}) sets verifying the required conditions.\medskip

\paragraph{Notations} $|S|$ denotes, according to the context, either the cardinality of the finite set $S$ or the length of the interval $S.$ All small letters, except $f,g, c, d, x, \alpha$ and $\theta,$ represent nonnegative integers. Let $x$ be a real number. We denote by $\left \lfloor{x}\right \rfloor$ the ``integer part" of $x$ and by $\{x\}$ the ``fractional part" of $x.$ Thus  $$x=\left \lfloor{x}\right \rfloor + \{x\}$$
 where $\left \lfloor{x}\right \rfloor$ is an integer and $0\leq\{x\}<1,$ this writing being {\it unique}.\\

In Section 2, we indicate or prove some properties of asymptotic density and in Section 3, we prove the next theorem.

\begin{thm}\label{thm1}
Let $\alpha$ be a real number, $0\leq\alpha\leq1,$ and $B$ a finite subset of $\mathbb{N}.$ Then there exists a set $A\subset \mathbb{N}$ such that $d(A+B)=\alpha.$
\end{thm}

Let us now suppose that $\alpha$ is given and $A=B.$ Our goal was to prove the existence of a set $A$ such that $d(2A)=\alpha.$ The proof(s) can be generalized to sumsets of more summands $A$ ; precisely, we prove the following more general result, showing that, for the constructed set $A,$ the density of the sumsets $jA,$ $1\leq j\leq k,$ increases regularly with $j.$ Proofs are of different form, according to the nature of $\alpha$ : rational or irrational. Here is the formulation of the theorem proved in Section 4.

\begin{thm}\label{thm2}
Let $\alpha$ be a real number, $0\leq\alpha\leq1,$ and $k$ an integer, $k\geq2.$ Then there exists a subset $A$ of $\mathbb{N}$ such that for every $j,$ $1\leq j\leq k,$ one has $$d(jA)={{j\alpha}\over k}.$$
In particular, $ d(kA)=\alpha.$
\end{thm}
\noindent
We give a  list of open questions and prospects for further research in Section 5.\\
\noindent

\section{Auxiliary results}
We mention next lemma without proof as it is a basic result concerning the density. Although the lemma is elementary in nature, it is quite useful for our study.
\begin{lem}
Let the set $X$ in  Definition \ref{definition1} be infinite, $X=:\{a_1<a_2<\dots\}.$ Then
$${\underline d}X:=\liminf_{x\rightarrow+\infty}{{X(x)}\over x}=\liminf_{k\rightarrow+\infty} {k\over{a_k}},$$ and
$${\overline d}X:=\limsup_{x\rightarrow+\infty}{{X(x)}\over x}=\limsup_{k\rightarrow+\infty} {k\over{a_k}}.$$
\end{lem}

As a consequence, we get the next property.

{\bf Property (i)} Let $X$ be a subset of $\mathbb {N}$ and $\theta$ a real number, $\theta\geq1.$ Put $\theta . X:=\{ \left \lfloor{\theta a}\right \rfloor ;a\in X\}.$ Then ${\underline d}(\theta . X)=\theta^{-1}~{\underline d} X$ and ${\overline d}(\theta . X)=\theta^{-1}~{\overline d} X.$ \smallskip

Put
 $\mathbb{N}_+:=\mathbb{N}\setminus\{0\}=\{1,2,\dots\}.$ The uniform distribution of the sequence $(\{n\theta\})_{n\in \mathbb{N}_+}$ when $\theta$ is irrational (\cite{KN}, p. 8; \cite{SP}, p. 2 - 72) amounts to:

{\bf Property (ii)} Let $I$ be a subinterval of $[0,1[$ and $\theta$ an irrational  real number. Then
$$d\{n\in \mathbb{N}_+;\{n\theta\}\in I\}=|I|.$$

As a corollary of the previous properties, we get:

\begin{cor}\label{cor1}
Let $I$ be a subinterval of  $[0,1[$ and $\theta$ an irrational  real number, $\theta>1.$ Then
$$d\{\lfloor{n\theta}\rfloor;n\in\mathbb{N}_+,\{n\theta\}\in I\}=|I|/\theta.$$
\end{cor}

\section{Proof of Theorem \ref{thm1}}
If $\alpha$ is 0 or 1, the answer is easy: take, respectively, $A$ to be the set of powers of 2 or $A=\mathbb{N}.$

\noindent
{\tt We shall suppose in the sequel that} $0<\alpha<1.$

\begin{lem}
Without loss of generality, we can suppose that $\min B=0.$
\end{lem}

\begin{proof} Suppose that the theorem is proved when $b_1:=\min B=0.$ Now let $b_1>0.$ Put $B':=B-b_1.$ By the case $\min B=0,$ there is $A'\subset\mathbb{N}$ such that $d(A'+B')=\alpha.$ Put $A_1:=A'-b_1\subset\{-b_1,-b_1+1,\dots,-1\}\cup \mathbb{N}$ and $A:=\{a\in A_1;a\geq0\}.$ Notice that
$$A+B\subset A_1+B=A'+B'\subset\mathbb{N}$$
so $A_1+B$ has asymptotic density $\alpha.$ Then notice that
$$(A_1+B)\setminus(A+B)\subset\cup_{k=-1}^{-b_1} \{k+B\}.$$
The set in the right hand member is finite. It yields that $d(A+B)=d(A_1+B)=\alpha$ which proves the Lemma.
\end{proof}
{\bf Continuation of the proof of Theorem \ref{thm1}.} We suppose from now on that $\min B=0.$ Let also $b:=\max B,$ $k:=|B|.$ Of course the theorem is trivial if $k=1.$ We suppose in the sequel that $k\geq2;$ consequently $b\ge1.$ The set $A=\{a_1,a_2,\dots\},$ $a_1<a_2<\cdots,$ will be defined recursively.
We define $a_1:=\min A$ in the following manner:
$$a_1:=\min\{a\ge0;~\max_{n\geq1}{{(a+B)(n)}\over n}\leq\alpha\}.\leqno(1)$$
{\it Remark on} (1). Notice that $a_1=0$ if and only if, for all $n\geq1,$ $B(n)\leq \alpha n.$ For given $a$ and all $n\geq1,$ let $f(n):={{(a+B)(n)}\over n}.$ The function $f$ takes nonnegative values and is decreasing when $n\ge{a+b};$ so, for fixed $a,$ the maximum exists. This maximum is attained for (at least) an $n=n_1,$ $a\le{n_1}\le{a+b},$ and verifies, when $a\neq 0,$
$$\max_{n\geq1}f(n)=f(n_1)\le{k\over{n_1}}\le{k\over a}.$$
So this maximum is less than or equal to $\alpha$ provided that $a\ge{k\over\alpha}.$ It follows that the minimum in formula (1) exists and hence $a_1$ is well defined and $a_1\le\lceil{k\over\alpha}\rceil.$

{\it Recursion.} Suppose that $a_1,a_2,\dots,a_m$  have been defined, and let $A_m:=\{a_1,a_2,\dots,a_m\}.$ We suppose that, for all $n\geq1,$
$${(A_m+B)(n)\over n}\leq\alpha.$$  Then we define $a_{m+1}$ in the following manner:
$$a_{m+1}:=\min\{a>{a_m};~\max_{n\geq a}{{((A_m\cup\{a\})+B)(n)}\over n}\leq\alpha\}.\leqno(2)$$
{\it Remark on} (2). Notice that $a_{m+1}=a_m+1$ if and only if $((A_m\cup\{a_m+1\})+B)(n)\leq \alpha n$  for all $n\geq{a_m+1}.$
Otherwise, $a_{m+1}>a_m+1.$
For fixed  $a>a_m,$  let $g(n):={{((A_m\cup\{a\})+B)(n)}\over n},$ for all $n\geq a. $ The function $g$ takes nonnegative values and is decreasing when $n\ge{a+b};$ so, for fixed $a,$ the maximum exists. This maximum is attained for (at least) an $n=n_{m+1},$ $a\le{n_{m+1}}\le{a+b},$ and verifies
$$\max_{n\geq a}g(n)=g(n_{m+1})\le{{k(m+1)}\over{n_{m+1}}}\le{{k(m+1)}\over a}.$$
So this maximum is less than or equal to $\alpha$ provided that $a\ge{{k(m+1)}\over\alpha}.$ It follows that the minimum in formula (2) exists and hence $a_{m+1}$ is well defined and $a_{m+1}\le\lceil{{k(m+1)}\over\alpha}\rceil.$ \\

Let us now observe that, for all $n\in\mathbb{N},$ we have
$$(A+B)(n)\le\alpha n. $$
This is clear when $n=0$ or $n<a_1.$ Otherwise, there is $m\ge1$ such that $a_m\le n<a_{m+1}.$ Then we have $(A+B)(n)=(A_m+B)(n)\le\alpha n.$

\noindent
This implies that ${\overline d}(A+B)\le\alpha$.\\

 It remains to prove that ${\underline d}(A+B)\ge\alpha.$ To do that, it is sufficient to prove that, for any $\varepsilon>0,$ we have ${\underline d}(A+B)\ge\alpha-\varepsilon.$

In what follows, we shall need the next property of the  counting function $n\mapsto(A+B)(n)$ of the set $A+B.$ In its (short) proof, we shall need the fact that $0\in B.$

{\bf Property (iii)} If $a\in A,$ $a>1,$ then $${{(A+B)(a)}\over a}\ge{{(A+B)(a-1)}\over {a-1}}.$$

\begin{proof} Let $y:=(A+B)(a-1).$ Since $a\in A$ and $0\in B,$ we get that $a=a+0\in A+B$ and so $(A+B)(a)=y+1.$ We have to show that $(y+1)/a\geq y/(a-1).$ The verification is straightforward and this proves Property (iii).
\end{proof}
{\bf Continuation of the proof of Theorem \ref{thm1}.} Let $0<\varepsilon<\alpha.$ We shall prove that  ${\underline d}(A+B)\ge\alpha-\varepsilon$ by contradiction. Suppose that ${\underline d}(A+B)<\alpha-\varepsilon.$ Then the set
$$S:=\{n\in\mathbb{N}; (A+B)(n)<(\alpha-\varepsilon)n\}$$ is infinite.

We observe that, for $0<\alpha<1,$ the constructed set $A$ is neither finite nor cofinite. $A$ is a collection of finite ``blocks", each block  consisting of one
or of  a finite number
of consecutive integers, and two consecutive blocks are separated by a ``hole" of length at least 2.

The preceding Property (iii) implies that if an element $a$ of $A$ , $a>1,$ belongs to $S$, then $a-1$ also belongs to $S.$ And since $a$ belongs to a block of $A,$ the last (the biggest) element of the hole just before the block to which $a$ belongs, is also an element of $S.$ We conclude that the set $S':=S\setminus A$ is infinite.

From $S'$ we can extract an infinite, strictly increasing, sequence of positive integers $(N_t)_{t\ge1}$ such that to each $t\ge1$ there  corresponds an index $m_t,$ verifying :

(i) $a_{m_t}<N_t<a_{m_t+1},$ and

(ii) $1\le{m_1} < m_2 < \cdots$ .

Let us fix now an index $t\ge1.$ Recall that
$$ (A+B)(N_t)=(A_{m_t}+B)(N_t)<(\alpha-\varepsilon)N_t~. \leqno(3)$$
Let $A':=A_{m_t}\cup\{N_t\}.$   By the formula (2), we get that
$$\max_{n\ge N_t}{1\over n}(A'+B)(n)>\alpha.$$
So there is an integer $n',$ $N_t\le n'\le N_t+b,$ such that
$$ (A'+B)(n')>\alpha n' ~.\leqno(4)$$
By the construction,
$$ (A_{m_t}+B)(n')\le \alpha n' ~.\leqno(5)$$
We observe that
$$ (A_{m_t}+B)(n')-(A_{m_t}+B)(N_t)\le n'- N_t\le b ~,$$
which implies, using (3), that
$$(A_{m_t}+B)(n')\le (A_{m_t}+B)(N_t)+b<(\alpha-\varepsilon)N_t+b ~.\leqno(6)  $$
We also observe that
$$ (A'+B)(n')-(A_{m_t}+B)(n')\le k ~,$$
which implies, using (4), that
$$ (A_{m_t}+B)(n')\ge(A'+B)(n')-k>\alpha n'-k\ge\alpha N_t-k~.\leqno(7)$$
The left members of (6) and (7) are the same. Comparing their right members, we get that $\varepsilon N_t < k+b.$

This is not true for any $t$, since $N_t$ tends to infinity. This implies that the hypothesis ${\underline d}(A+B)<\alpha-\varepsilon$ is false and completes the proof of Theorem \ref{thm1}.

\section{Proof of Theorem \ref{thm2}}
If $\alpha$ is 0 or 1, the answer is easy: take, respectively, $A$ to be the set of powers of 2 or $A=\mathbb{N}.$

\noindent
{\tt We shall suppose in the sequel that} $0<\alpha<1.$\\

We distinguish two cases

\noindent
$\clubsuit$ {\bf Case A}: $\alpha$ rational; say, $\alpha={m\over n},$ where $m,n$ are integers, $1\leq m\leq n-1.$ It is not necessary that $\gcd(m,n)=1.$ Our construction of such a set $A$ is {\tt simpler} (see also the remark at the end of the proof of Case A) when $m\geq3.$ So multiplying, if necessary, the terms of the fraction for $\alpha$ by 2 or by 3,
{\tt we will suppose in the sequel that}\;\;  $3\leq m\leq n-1.$\smallskip

Let $H:=\{0,1,\dots,m-2,m\}.$ We shall prove that the set
$$A:=\cup_{h\in H}(nk\cdot \mathbb{N}+h)$$ verifies $d(jA)=j\alpha/k,$ for all $j,$ $1\leq j\leq k.$ To prove that, let us first observe that
$$jA=nk\cdot \mathbb{N}+jH;$$ hereon it is easy to verify that each element of the left member belongs to the right member and vice versa.

Notice that $k~{\max H}=km<nk.$ The set $jA$ being a finite union of mutually disjoint arithmetic progressions of difference $nk,$ we have
$$d(jA)=\sum_{t\in jH}d(nk\cdot \mathbb{N}+t)=|jH|{1\over{nk}}.$$
But $jH=\{0,1,\dots,jm-2,jm\}$ because $jm=j~{\max H}\in jH,$ $jm-1\not\in jH,$ and every nonnegative integer less than $jm-1$ belongs to $jH.$ For example, $jm-2=(j-1)m+(m-2)\in jH$ ; or $jm-3=(j-1)m+(m-3)\in jH.$ So
 $$d(jA)=|jH|{1\over{nk}}={{jm}\over{nk}}={{j\alpha}\over k}~,~ 1\leq j\leq k~.$$
 \begin{rem}
 It is possible to invent specific constructions for $m=1$ and for $m=2.$
 \end{rem}

 \noindent
$\clubsuit$ {\bf  Case B}: $\alpha$ irrational, $0<\alpha<1.$

\noindent
We put $\theta:=1/\alpha.$ We recall our notation  $\mathbb{N}_+:=\mathbb{N}\setminus\{0\}=\{1,2,\dots\}.$ We shall prove that the set $$A:=\{\lfloor{n\theta}\rfloor;n\in \mathbb{N}_+,  \{n\theta\}<{1\over k}\}$$
verifies $d(jA)=j\alpha/k,$ for all $j,$ $1\leq j\leq k.$

For $j=1,$ this follows from Corollary \ref{cor1}.\\
 {\tt We suppose in the sequel that} $2\leq j\leq k.$\\

${\spadesuit}$ We firstly prove that ${\overline d}(jA)\leq j\alpha/k.$\\
\noindent
To do that, we begin by proving that $jA\subset T_j$ where
$$T_j:=\{\lfloor{m\theta}\rfloor; m\in \mathbb{N}_+,  \{m\theta\}<j/k \}.\leqno (8)$$
An element of $jA$ is of the form $a_1+\cdots+a_j$ where $a_i=\lfloor{n_i\theta}\rfloor\in A,$ $1\leq i\leq j.$ This yields

$a_1+\cdots+a_j=\lfloor{n_1\theta}\rfloor
+\cdots+\lfloor{n_j\theta}\rfloor={n_1\theta}
-\{n_1\theta\}+\cdots+{n_j\theta}-\{n_j\theta\}$

\noindent and consequently

$(n_1+\cdots+n_j)\theta=a_1+\cdots+a_j + \{n_1\theta\}+\cdots+\{n_j\theta\}.$

\noindent We have

$0<\{n_1\theta\}+\cdots+\{n_j\theta\}<j{1\over k}\leq1$,

\noindent and so $\{n_1\theta\}+\cdots+\{n_j\theta\}$ is the fractional part and $a_1+\cdots+a_j$ is the integer part of $(n_1+\cdots+n_j)\theta.$ In other terms and according to Definition (8), $a_1+\cdots+a_j$ belongs to $T_j.$ Since, by Corollary \ref{cor1}, $T_j$ has asymptotic density $j\alpha/k,$ the desired inequality follows.\smallskip

${\spadesuit}$  We shall now prove that ${\underline d}(jA)\geq j\alpha/k.$\\
\noindent
This will be done in the following way. We fix a real number $\varepsilon,$ $0<\varepsilon<{1\over{4k}},$ and we shall prove that
$${\underline d}(jA)\geq\alpha({j\over k}-\varepsilon).\leqno (9)$$ Taking the limit for
$\varepsilon$ tending to zero in (9) gives the desired inequality for ${\underline d}(jA).$

To prove (9), we introduce the set
$$B:=\{\lfloor{N\theta}\rfloor;N\in \mathbb{N}_+,{\varepsilon\over2}\leq\{N\theta\}<{j\over k}-{\varepsilon\over2}\}\leqno (10)$$ which, by Corollary \ref{cor1}, has asymptotic density $({j\over k}-{\varepsilon\over2}-{\varepsilon\over2})/\theta=\alpha({j\over k}-\varepsilon)$ and we will verify that {\tt almost all} (that is, all except a finite number) {\tt elements of} $B$ {\tt belong to} $jA;$ this implies (9).

Let $\ell:={j\over k}-\varepsilon,$ which is a positive real number $(\ell>{2\over k}-{1\over{4k}}>0)$ less than 1 $(\ell\leq{k\over k}-\varepsilon<1).$ We split the interval $[{\varepsilon\over2},{j\over k}-{\varepsilon\over2})$ into $j$ intervals of equal length $\ell/j$ :
$$[{\varepsilon\over2},{j\over k}-{\varepsilon\over2})=\cup_{i=0}^{j-1}I_i $$ where
$$I_i:=[{\varepsilon\over2}+{i\ell\over j},{\varepsilon\over2}+{(i+1)\ell\over j})~~,~~0\leq i\leq j-1~.$$
The set $B$ splits into $j$ sets $B=\cup_{i=0}^{j-1}B_i$, where $B_i:=\{\lfloor{N\theta}\rfloor;N\in \mathbb{N}_+, \{N\theta\}\in I_i\},$ and it will be sufficient (and necessary!) to prove that, for each $i,$ all except finitely many elements of $B_i$ lie in $jA.$ Here is the procedure.

First, one can easily verify that
$$0 \le{\varepsilon\over2}+{(i+1)\ell\over j}-{1\over k}<{\varepsilon\over2}+{i\ell\over j}.$$
By the uniform distribution modulo 1 of the sequence $(\{n\theta\})_n$ (Property (ii)), there is a positive integer $m_i$ such that
$${\varepsilon\over2}+{(i+1)\ell\over j}-{1\over k}<(j-1)\{m_i\theta\}<
{\varepsilon\over2}+{i\ell\over j}.\leqno (11)$$ We have that $\lfloor{m_i\theta}\rfloor$ belongs to $A$ since
$$\{m_i\theta\}<{1\over{j-1}}({\varepsilon\over2}+{i\ell\over j})\leq{1\over{j-1}}({\varepsilon\over2}+{(j-1)\ell\over j})={\varepsilon\over{2(j-1)}}+{1\over j}\ell=$$ $$={\varepsilon\over{2(j-1)}}+{1\over j}({j\over k}-\varepsilon)={1\over k}-\varepsilon({1\over j}-{1\over{2(j-1)}})= {1\over k}-\varepsilon{{j-2}\over{2j(j-1)}}\leq {1\over k}.$$
We shall prove that, for every $N>(j-1)m_i$ such that $\lfloor{N\theta}\rfloor$ belongs to $B_i,$ $\lfloor{N\theta}\rfloor$ belongs also to $jA.$

Since $\lfloor{N\theta}\rfloor$ belongs to $B_i,$ we have that
$${\varepsilon\over2}+{i\ell\over j}\leq\{N\theta\}<{\varepsilon\over2}+{(i+1)\ell\over j}.\leqno(12)$$
Putting (11) and (12) together, gives
$$0<\{N\theta\}-(j-1)\{m_i\theta\}<{1\over k}. \leqno (13)$$
From the equality
$N\theta=(N-(j-1)m_i)\theta+(j-1)m_i\theta,$ taking the integer part and the fractional part of each multiple of $\theta$, we get
$$\lfloor{N\theta}\rfloor+ \{N\theta\}-(j-1)\{m_i\theta\}=
 \lfloor{(N-(j-1)m_i)\theta}\rfloor+\{(N-(j-1)m_i)\theta\}+(j-1)\lfloor{m_i\theta}\rfloor. \leqno(14)$$

By the uniqueness of decomposition of a real number into its integer part and its fractional part, the inequality (13) implies that the fractional parts appearing in (14) verify

$$\{N\theta\}-(j-1)\{m_i\theta\}=\{(N-(j-1)m_i)\theta\} \leqno(15)$$ and this, combined with (14), gives that
$$\lfloor{N\theta}\rfloor=  \lfloor{(N-(j-1)m_i)\theta}\rfloor +(j-1)\lfloor{m_i\theta}\rfloor.   \leqno(16)$$ As observed before, $\lfloor{m_i\theta}\rfloor\in A.$ Because of (15) and (13), $\lfloor{(N-(j-1)m_i)\theta}\rfloor$ belongs to $A$ and (16) gives us that $\lfloor{N\theta}\rfloor\in jA.$ This completes the proof of (9) and of the whole Theorem \ref{thm2}.\\

{\bf Added in proof.-} In \cite{V} the author resolves in a more general context ($\mathbb{Z}^t$ instead of $\mathbb{N}$) a problem which, in some sense, contains as special case the problem solved in the above theorem. When $k=2,$ the meaning of our sentence ``in some sense" is as follows: Given two positive real numbers $\alpha_1$ and
 $\alpha_2$ such that  $\alpha_1+\alpha_2\leq1$ and a third real number $\gamma,$
$\alpha_1+\alpha_2\leq\gamma\leq1,$
Bodo Volkmann \cite{V} constructs sets $A_1, A_2$ of natural numbers satisfying $d(A_1)=\alpha_1$, $d(A_2)=\alpha_2$ and $d(A_1+A_2)=\gamma.$ The construction uses ideas similar to the ours. The principle is the same: to use uniform distribution of fractional parts in order to obtain sets of integers with prescribed density. In the case when $\alpha_1=\alpha_2,$ the sets $A_1,A_2$ are not equal. The set $A_1$ is constructed in a way similar to the one used in our proof, but the set $A_2$ is constructed in a specific way in order to obtain  $d(A_1+A_2)=\gamma.$ Even in the case when $\gamma=2\alpha_1=2\alpha_2$, the set $A_2$ is different from $A_1.$ We observed that in this particular case Volkmann's construction can be slightly modified to give $A_2=A_1$ thus providing another proof, based on Lemma 2 of \cite{V}, of our theorem. A similar
remark is valid when $k\geq3.$

\section{Future prospects}

We separate this section into questions: {\it Q1} to {\it Q7}.\smallskip

\paragraph{Q1 - Thin sets} Concerning the case $B=A,$ it would be interesting to study the existence of thin sets $A$ verifying $d(kA)=\alpha.$ For {\it additive bases} $A$ (that is, for sets $A$ verifying $kA= \mathbb{N}$ for some $k,$ called ``order" of the basis $A)$ this was done by Cassels \cite{C} (see also \cite{HR}, p. 35-43) where, for every $k\geq2,$ a ``thin" basis of order $k$ was found. As in the case of additive bases, in our situation, with $\alpha>0,$ the condition $A(x)\geq c x^{1/k}$ is necessary. In Cassels' construction, the basis corresponding to the order $k$ verifies $A(x)\leq c' x^{1/k}.$ The question here is to find sets $A$ verifying $d(kA)=\alpha>0$ and $A(x)\leq c'' x^{1/k}.$ Here is a related question: is it possible to extract from Cassels' thin basis $A$ (such that $kA= \mathbb{N})$ a set $A'\subset A$ such that $d(kA')=\alpha?$ If this is possible, the condition $A(x)\leq c'' x^{1/k}$ is automatically verified.

{\bf Remark.-} The above mentioned Cassels' thin bases allow to answer the above question when $\alpha=1/n.$ Take the Cassels asymptotic basis $C=\{c_1<c_2<\dots\}$ \cite{C} (see also \cite{HR} , p. 37) of order $n$. It verifies $c_m=\beta m^n+O(m^{n-1}).$ Now a solution to the above asked question is to take $A:=\{nx;x\in C\}.$ For other values of $\alpha$ the question remains open.

Concerning the case of a given finite set $B$ considered in this paper, we see two questions:
\begin{enumerate}

\item To evaluate the density (or the upper and the lower densities) of the greedily constructed
 set $A.$ The constructed set seems to be ``the thickest".

\item To determine the more thin set $A$ that verifies $d(A+B)=\alpha$: A necessary condition is that ${\underline d}A\geq\alpha/|B|.$

\end{enumerate}

\paragraph{Q2 - Thick sets} In the case of bases, the thickest set $A$ verifying $kA=\mathbb{N}$ is $A=\mathbb{N}.$ What are thick sets $A$ satisfying $d(kA)=\alpha$ when $\alpha<1?$ In the case $\alpha={1\over r},$ where $r$ is an integer, $r\geq 2,$ the answer is trivial: take $A=\{0,r,2r, 3r,\dots\}.$ But in general the answer is not obvious. It may depend on the nature of $\alpha$: rational or irrational.

\paragraph{Q3 - $B$ infinite} What happens for given infinite $B?$ Find necessary or sufficient (or both!) conditions on $B$ (on the upper and lower densities of $B$) such that $A$ exists. Our greedy method of Section 2, with some supplementary considerations on formulas (1) and (2), allows to construct a set $A$ such that ${\overline d}(A+B)\leq\alpha$ when $dB=0.$ But we are not able to prove that $d(A+B)=\alpha;$ nor to disprove it for a specific set $B.$

\paragraph{Q4 - Other densities} Replace asymptotic density by other densities. See \cite{G} for a list of definitions. For instance, the {\it exponential density}, defined as [compare with  Definition 1.1 of asymptotic density in Section 1]
$${\varepsilon}X:=\lim_{x\rightarrow+\infty}{{\log X(x)}\over {\log x}},$$
 could be of interest. That is, given $B$ and $\alpha,$ is there $A$ such that the exponential density of $A+B$ is equal to $\alpha?$ Instead of using specific (concrete) definitions of density, one could use axiomatically defined densities which generalize some of the usual concepts of density; see \cite{FS}, \cite{LT1} and \cite{LT2}. Concerning the analog of Theorem 1.3, the cases of Schniremann density and of {\it lower} asymptotic density were already considered (but only as for existence of a set $A$) in \cite{Le}, \cite{Ch} and \cite{Nat}, where best possible results to Mann's and Kneser's theorems are proved.

\paragraph{Q5 - Couple of densities} It is possible to consider the initial problem and to ask all the above questions replacing $\alpha$ by two real numbers $\alpha'$ and $\alpha'',$ $0\leq\alpha'\leq\alpha''\leq1,$ that will be, respectively, the lower and the upper densities. For instance, given $\alpha'$ and $\alpha''$ as above and $k\geq2,$ find a set $A$ such that ${\underline d}(kA)=\alpha'$ and ${\overline d}(kA)=\alpha''.$

\paragraph{Q6 - Generalization} Given a subset $B$ of $\mathbb{N}$, finite or of zero asymptotic density, a real number $\alpha$, $0\leq\alpha\leq1,$ and an integer $k\geq2$, is there a set $A\subset \mathbb{N}$ such that $d(B+kA)=\alpha$? Search for ``thin" and for ``thick" such sets $A.$

\paragraph{Q7 - Last but not least: $A$ and $B$} The more studied question on classical additive complements is to compare the functions $(A+B)(x)$ and $A(x)B(x).$ Obviously $(A+B)(x)\leq A(x)B(x).$ So to have ``thin" sets $A$ and $B$ means that
$A(x)B(x)$ is not much bigger than $(A+B)(x).$ In the classical case, $(A+B)(x)$ is equal, up to a constant, to $x.$ In our case, with $\alpha>0,$ $(A+B)(x)$ is ``equivalent" to $\alpha x:$ $\lim_{x\rightarrow+\infty}(A+B)(x)/{\alpha x}=1.$ So questions studied in
\cite{Nar},  \cite{D} or \cite{FC} and finally in  \cite{R} can be formulated with $\alpha x$ in place of $x.$

\normalsize
\baselineskip=17pt

\subsection*{Acknowledgements}
The authors are thankful

- to Salvatore Tringali for having suggested the idea used in the proof of Case A of Theorem \ref{thm2};

- to H{\'e}di Daboussi, W\l{}adys\l{}aw Narkiewicz, Jo{\"e}l Rivat, Andrzej Schinzel and Bodo Volkmann for fruitful discussions;

- to the anonymous referee for useful observations on the initial version.


\end{document}